\documentclass{article}

\usepackage{authblk}
\usepackage{amssymb}
\usepackage{amsmath}
\usepackage{amsthm}
\newtheorem{lem}{Lemma}[section]
\newtheorem{thm}{Theorem}[section]

\newtheorem{prop}{Proposition}[section]
\theoremstyle{definition}

\theoremstyle{example}
\newtheorem{example}{Example}[section]

\makeatletter
\def\blfootnote{\xdef\@thefnmark{}\@footnotetext}
\makeatother

\begin{document}
\title{Center of $\mathcal{H}_R(0,c_{\tilde{s}})$ associated to a pro-$p$-Iwahori Weyl group}\blfootnote{This paper is supported by NSF grant DMS-1463852.}
\author{Yijie Gao}
\affil{Deparment of Mathematics \\ University of Maryland, College Park}
\date{}
\maketitle

\begin{abstract}
Let $W$ be an Iwahori Weyl group and $W(1)$ be an extension of $W$ by an abelian group. In \cite{vm1}, Vigneras gave a description of the $R$-algebra $\mathcal{H}_R(q_{\tilde{s}},c_{\tilde{s}})$ associated to $W(1)$, and in \cite{vm2}, Vigneras gave a basis of the center of $\mathcal{H}_R(q_{\tilde{s}},c_{\tilde{s}})$ using the Bernstein presentation of $\mathcal{H}_R(q_{\tilde{s}},c_{\tilde{s}})$. In this paper, we restrict to the case where $q_{\tilde{s}}=0$ and use the Iwahori-Matsumoto presentation to give a basis of the center of $\mathcal{H}_R(0,c_{\tilde{s}})$.
\end{abstract}

\section{Introduction}

Iwahori-Hecke algebras are deformations of the group algebras
of finite Coxeter groups $W$ with nonzero parameters. They play an important role in the study of representations of finite groups of Lie type. In \cite{gr}, Geck and Rouquier gave a basis of the center of Iwahori-Hecke algebras. The basis is closely related to minimal length elements in the conjugacy classes of $W$.

The 0-Hecke algebra was used by Carter and Lusztig in \cite{cl} in
the study of $p$-modular representations of finite groups of Lie type. $0$-Hecke algebras are deformations of the group algebras of finite Coxeter groups with zero parameter. In \cite{xh}, He gave a basis of the center of $0$-Hecke algebras associated to finte Coxeter groups. The basis is closely related to maximal length elements in the conjugacy classes of $W$.

Affine Hecke algebras are deformations of the group algebras of affine Weyl groups $W^{\text{aff}}$. They appear naturally in the representation theory of reductive $p$-adic groups. In \cite{gl89}, Lusztig gave a basis of the center of affine Hecke algebras. In \cite{xh}, He also mentioned a similar proof could be applied to give a basis of the center of affine $0$-Hecke algebras. The basis is closely related to finite conjugacy classes in $W^{\text{aff}}$.

Let $\mathbf{G}$ be a connected reductive group over a $p$-adic field $F$. The study of mod-$p$ representations of $\mathbf{G}(F)$ naturally involves the pro-$p$-Iwahori Hecke algebra of $\mathbf{G}(F)$. Let $R$ be a commutative ring. In \cite{vm1}, Vigneras discussed the $R$-algebra $\mathcal{H}_R(q_{\tilde{s}},c_{\tilde{s}})$ which generalizes the pro-$p$-Iwahori Hecke algebra of $\mathbf{G}(F)$. In \cite{vm2}, Vigneras gave a basis of the center of $\mathcal{H}_R(q_{\tilde{s}},c_{\tilde{s}})$ by using the Bernstein relation and alcove walks (the definition of alcove walk can be found in \cite{ug}). The basis of center is closely related to finite conjugacy classes in $W(1)$.

In general, the expression of the center in \cite{vm2} is complicated if we want to write it out explicitly by Iwahori-Matsumoto presentation. But for $R$-algebras $\mathcal{H}_R(0,c_{\tilde{s}})$, we can give an explicit description of the center by Iwahori-Matsumoto presentation. This is the main result of this paper. In Section $2$, we review the definition of $\mathcal{H}_R(q_{\tilde{s}},c_{\tilde{s}})$ and some properties of the group $W(1)$. In Section $3$, we define a new operator $r_{v,w}$. In Section $4$, we show maximal length terms of a central element in $\mathcal{H}_R(0,c_{\tilde{s}})$ comes from finite conjugacy classes in $W(1)$. In Section $5$, we prove some technical results regarding $r_{v,w}$, where $w$ is in some finite conjugacy class. In Section $6$, we give a basis of the center of $\mathcal{H}_R(0,c_{\tilde{s}})$. In Section $7$, we give some examples to show how the main result works. 

\section{Preliminary}

The symbols $\mathbb{N},\mathbb{Z},\mathbb{R}$ refers to the natural numbers, the integers and the real numbers.

Let $\Sigma$ be a reduced and irreducible root system with simple system $\Delta$. Let $W_0$ be the finite Weyl group of $\Sigma$, and $S_0$ be the set of simple reflections corresponding to $\Delta$. Then $S_0$ is a generating set of $W_0$.

Let $\mathcal{V}=\mathbb{Z}\Sigma^{\vee}\otimes_{\mathbb{Z}}\mathbb{R}$ be the $\mathbb{R}$-vector space spanned by the dual root system $\Sigma^{\vee}$. Let $\Sigma^{\text{aff}}$ be the affine root system associated to $\Sigma$, i.e. the set $\Sigma+\mathbb{Z}$ of affine functionals on $\mathcal{V}$. The term hyperplane always means the null-set of an element of $\Sigma^{\text{aff}}$.

Choose a special vertex $\mathfrak{v}_0\in\mathcal{V}$ such that $\mathfrak{v}_0$ is stabilized by the action of $W_0$. Let $\mathfrak{C}_0$ be the Weyl chamber at $\mathfrak{v}_0$ corresponding to $S_0$ and let $\mathfrak{A}_0\in\mathfrak{C}_0$ be the alcove for which $\mathfrak{v}_0\in\bar{\mathfrak{A}}_0$ where $\bar{\mathfrak{A}}_0$ is the closure of $\mathfrak{A}_0$. 

Let $W^{\text{aff}}$ be the affine Weyl group of $\Sigma^{\text{aff}}$ and $S^{\text{aff}}$ be the set of affine reflections corresponding to walls of $\mathfrak{A}_0$. Then $S^{\text{aff}}$ is a generating set of $W^{\text{aff}}$ extended from $S_0$. Denote by $\ell: W^{\text{aff}}\rightarrow\mathbb{N}$ the length function relative to the generating set $S^{\text{aff}}$. The group $W^{\text{aff}}$ is also equipped with the Bruhat order $\leq$.

Let $F$ be a non-archimedean local field and let $\mathbf{G}$ be a connected reductive $F$-group. Let $\mathbf{T}\subseteq \mathbf{G}$ be a maximal $F$-split torus and set $\mathbf{Z}$ and $\mathbf{N}$ be $\mathbf{G}$-centralizer and $\mathbf{G}$-normalizer of $\mathbf{T}$ respectively. Let $\mathbf{G}(F),\mathbf{T}(F),\mathbf{Z}(F),\mathbf{N}(F)$ be the groups of $F$-points of $\mathbf{G},\mathbf{T},\mathbf{Z},\mathbf{N}$.  Then the group $\mathbf{Z}(F)$ admits a unique parahoric subgroup $\mathbf{Z}(F)_0$. We may define the Iwahori-Weyl group of $(\mathbf{G},\mathbf{T})$ to be the quotient $W:=\mathbf{N}(F)/\mathbf{Z}(F)_0$.

There are two ways to express the Iwahori-Weyl group as a semidirect product. By the work of Bruhat and Tits, it is known that there exists a reduced root system $\Sigma$ such that the corresponding affine Weyl group is a subgroup of $W$. Denoting by $W_0$ the finite Weyl group of $\Sigma$, it can be shown that $W=\Lambda\rtimes W_0$ and that $W=W^{\text{aff}}\rtimes\Omega$. The action of $W^{\text{aff}}$ on $\mathcal{V}$ extends to an action of $W$. The subgroup $\Lambda$ acts on $\mathcal{V}$ by translations and the subgroup $\Omega$ acts on $\mathcal{V}$ by invertible affine transformations that stabilize the base alcove $\mathfrak{A}_0$ in $\mathcal{V}$. 

The group $\Omega$ stabilizes $S^{\text{aff}}$. We can extend the length function $\ell: W^{\text{aff}}\rightarrow\mathbb{N}$ to $W$ by inflation along the projection $W^{\text{aff}}\rtimes\Omega\rightarrow W^{\text{aff}}$. Then the subgroup of length $0$ elements in $W$ is $\Omega$. The Bruhat order on $W$ can also be defined. Let $v=v'\tau,w=w'\tau'$ be two elements in $W$ where $v',w'\in W^{\text{aff}}$ and $\tau,\tau'\in\Omega$, then $v\leq w$ if and only if $v\leq w$ and $\tau=\tau'$. 

The group $\Lambda$ is finitely generated and abelian and the action of $\Lambda$ on $\mathcal{V}$ is given by the homomorphism $$\nu: \Lambda\rightarrow\mathcal{V}$$ such that $\lambda\in\Lambda$ acts as translation by $\nu(\lambda)$ in $\mathcal{V}$. The group $\Lambda$ is normalized by $x\in W_0$: $x\lambda x^{-1}$ acts as translation by $x(\nu(\lambda))$. The length $\ell$ is constant on each $W_0$-conjugacy class in $\Lambda$. By Lemma $2.1$ in \cite{vm2}, a conjugacy class of $W$ is finite if and only if it is contained in $\Lambda$, and infinite if and only if it is disjoint from $\Lambda$.

We'll later use the following geometric characterization of length (see Lemma $5.1.1$ in \cite{sr}):
\begin{lem}\label{13}
Let $w\in W$ and $s\in S^{\text{aff}}$. If $H_{s}$ is the hyperplane stabilized by $s$, then
\begin{itemize}
\item
$\ell(sw)>\ell(w)$ if and only if $\mathfrak{A}_0$ and $w(\mathfrak{A}_0)$ are on the same side of $H_s$,
\item
$\ell(ws)>\ell(w)$ if and only if $\mathfrak{A}_0$ and $w(\mathfrak{A}_0)$ are on the same side of $w(H_s)$.
\end{itemize}
\end{lem}

We will also use the following result on Bruhat order.

\begin{lem}\label{14}
Let $x,y\in W$ with $x\leq y$. Let $s\in S^{\text{aff}}$. Then
\begin{itemize}
\item
$\min\{x,sx\}\leq\min\{y,sy\}$ and $\max\{x,sx\}\leq\max\{y,sy\}$.
\item
$\min\{x,xs\}\leq\min\{y,ys\}$ and $\max\{x,xs\}\leq\max\{y,ys\}$.
\end{itemize}
\end{lem}

\begin{proof}
When $\Omega$ is trivial, this is all well-known: see Corollary $2.5$ in \cite{gl03}. The more general statement is immediate by definition of the Bruhat order on $W$ because $W=W^{aff}\rtimes\Omega$.
\end{proof}

Let $Z$ be an arbitrary abelian group, and let $W(1)$ be an extension of $W$ by $Z$ given by the short exact sequence $$1\rightarrow Z\rightarrow W(1)\stackrel{\pi}\rightarrow W\rightarrow 1.$$
We denote by $X(1)$ the inverse image in $W(1)$ of a subset $X\subseteq W$. Then $W^{\text{aff}}(1),\Lambda(1)$ are normal in $W(1)$ and $W(1)=W^{\text{aff}}(1)\Omega(1)=\Lambda(1)W_0(1),Z=W^{\text{aff}}(1)\cap\Omega(1)=\Lambda(1)\cap W_0(1)$. The length function on $W$ inflates to a length function on $W(1)$, still denoted by $\ell$, such that $\ell(\tilde{w})=\ell(\pi(\tilde{w}))$ for $\tilde{w}\in W(1)$.

Let $R$ be a commutative ring. For $\tilde{w}\in W(1)$ and $t\in Z$, $\tilde{w}(t)=\tilde{w}t\tilde{w}^{-1}$ depends only on the image of $\tilde{w}$ in $W$ because $Z$ is commutative. By linearity the conjugation defines an action 
$$(\tilde{w},c)\mapsto \tilde{w}(c): W(1)\times R[Z]\rightarrow R[Z]$$
of $W(1)$ on $R[Z]$ factoring through the map $\pi: W(1)\rightarrow W$.

We recall the definition of the generic algebra $\mathcal{H}_R(q_{\tilde{s}},c_{\tilde{s}})$ introduced in \cite{vm1}.
\begin{thm}
Let $(q_{\tilde{s}},c_{\tilde{s}})\in R\times R[Z]$ for all $\tilde{s}\in S^{\text{aff}}(1)$. Suppose
\begin{itemize}
\item
$q_{\tilde{s}}=q_{\tilde{s}t}=q_{\tilde{s'}}$,
\item
$c_{\tilde{s}t}=c_{\tilde{s}}t$ and $\tilde{w}(c_{\tilde{s}})=c_{\tilde{w}\tilde{s}\tilde{w}^{-1}}=c_{\tilde{s'}},$
\end{itemize}
where $t\in Z,\tilde{w}\in W(1),\tilde{s},\tilde{s'}\in S^{\text{aff}}(1)$ and $\tilde{s'}=\tilde{w}\tilde{s}\tilde{w}^{-1}$.

Then the free $R$-module $\mathcal{H}_R(q_{\tilde{s}},c_{\tilde{s}})$ of basis $(T_{\tilde{w}})_{\tilde{w}\in W(1)}$ admits a unique $R$-algebra structure satisfying
\begin{itemize}
\item
the braid relations: $T_{\tilde{w}}T_{\tilde{w'}}=T_{\tilde{w}\tilde{w'}}$ for $\tilde{w},\tilde{w'}\in W(1), \ell(\tilde{w})+\ell(\tilde{w'})=\ell(\tilde{w}\tilde{w'})$,
\item
the quadratic relations: $T_{\tilde{s}}^2=q_{\tilde{s}}T_{\tilde{s}^2}+c_{\tilde{s}}T_{\tilde{s}}$ for $\tilde{s}\in S^{\text{aff}}(1)$,
\end{itemize}
where $c_{\tilde{s}}=\sum_{t\in Z}c_{\tilde{s}}(t)t\in R[Z]$ is identified with $\sum_{t\in Z}c_{\tilde{s}}(t)T_t$.
\end{thm}

The algebra $\mathcal{H}_R(q_{\tilde{s}},c_{\tilde{s}})$ is called the $R$-algebra of $W(1)$ with parameters $(q_{\tilde{s}},c_{\tilde{s}})$.

The basis of the center of $\mathcal{H}_R(q_{\tilde{s}},c_{\tilde{s}})$ given in \cite{vm2} can be very complicated when written explicitly by Iwahori-Matsumoto presentation. But when $q_{\tilde{s}}=0$, we can write out the basis explicitly. In this paper, all our discussions are under the condition of $q_{\tilde{s}}=0,\forall \tilde{s}\in S^{\text{aff}}(1)$, that is, the algebra $\mathcal{H}_R(0,c_{\tilde{s}})$. In this case, the quadratic relations become $T_{\tilde{s}}^2=c_{\tilde{s}}T_{\tilde{s}}$.

For convenience, we define a $W(1)$-action on $\mathcal{H}_R(0,c_{\tilde{s}})$ given by $\tilde{w}\bullet T_{\tilde{w'}}=T_{\tilde{w} \tilde{w'} \tilde{w}^{-1}}$ for any $\tilde{w}, \tilde{w'}\in W(1)$, extended linearly to all elements in $\mathcal{H}_R(0,c_{\tilde{s}})$.

The following lemma is useful in later discussion:
\begin{lem}\label{1}
Let $\tilde{w}_1,\tilde{w}_2,\tilde{v}_1,\tilde{v}_2\in W(1), \tilde{s}_1,\tilde{s}_2\in S^{\text{aff}}(1)$, and suppose $\tilde{w}_1 \tilde{s}_1 \tilde{v}_1=\tilde{w}_2 \tilde{s}_2 \tilde{v}_2$ and $\pi(\tilde{w}_1 \tilde{v}_1)=\pi(\tilde{w}_2 \tilde{v}_2)$. Then $\tilde{w}_1 c_{\tilde{s}_1} \tilde{v}_1=\tilde{w}_2 c_{\tilde{s}_2}\tilde{v}_2$.
\end{lem}

\begin{proof}
Since $\pi(\tilde{w}_1 \tilde{v}_1)=\pi(\tilde{w}_2 \tilde{v}_2)$, we have $\tilde{w}_2 \tilde{v}_2=\tilde{w}_1 t \tilde{v}_1$ for some $t\in Z$, hence $\tilde{w}_1^{-1} \tilde{w}_2=t \tilde{v}_1 \tilde{v}_2^{-1}$. Then $\tilde{s}_1=\tilde{w}_1^{-1}\tilde{w}_2 \tilde{s}_2 \tilde{v}_2 \tilde{v}_1^{-1}=t(\tilde{v}_1 \tilde{v}_2^{-1})\tilde{s}_2(\tilde{v}_1 \tilde{v}_2^{-1})^{-1}$, so $c_{\tilde{s}_1}=t(\tilde{v}_1 \tilde{v}_2^{-1})c_{\tilde{s}_2}(\tilde{v}_1 \tilde{v}_2^{-1})^{-1}=\tilde{w}_1^{-1} \tilde{w}_2 c_{\tilde{s}_2} \tilde{v}_2 \tilde{v}_1^{-1}$, i.e., $\tilde{w}_1 c_{\tilde{s}_1} \tilde{v}_1=\tilde{w}_2 c_{\tilde{s}_2} \tilde{v}_2$.
\end{proof}

\section{A New Operator}

In this section, we will define an operator $r_{v,w}$ for any pair $(v,w)\in W\times W$ with $v\leq w$. This operator is the main ingredient of this paper.

For every $s\in S^{\text{aff}}$, pick a lifing $\tilde{s}$ in $S^{\text{aff}}(1)$, and for every $\tau\in\Omega$, pick a lifting $\tilde{\tau}$ in $\Omega(1)$. Let $w\in W$ with $\ell(w)=n$ and $\underline{w}=s_{i_1}s_{i_2}\dotsm s_{i_n}\tau$ be a reduced expression of $w$. A subexpression of $\underline{w}$ is a word $s_{i_1}^{e_{i_1}}s_{i_2}^{e_{i_2}}\dotsm s_{i_n}^{e_{i_n}}\tau$ with $(e_{i_1},e_{i_2},\dotsm,e_{i_n})\in\{0,1\}^n$. A subexpression is called non-decreasing if $\ell(s_{i_1}^{e_{i_1}}s_{i_2}^{e_{i_2}}\dotsm s_{i_n}^{e_{i_n}}\tau)=\sum_{k=1}^n e_{i_k}$. Let $v\leq w$, then there exists $(e_{i_1},e_{i_2},\dotsm,e_{i_n})\in\{0,1\}^n$ such that $\underline{v}_{\underline{w}}=s_{i_1}^{e_1}s_{i_2}^{e_2}\dotsm s_{i_n}^{e_n}\tau$ equals $v$ and is also a non-decreasing subexpression of $\underline{w}$. Let $\tilde{w}\in W(1)$ be a lifting of $w$, then $\tilde{w}$ has an expression $\underline{\tilde{w}}=t\tilde{s}_{i_1}\tilde{s}_{i_2}\dotsm\tilde{s}_{i_n}\tilde{\tau}$ for some $t\in Z$. Then the operator,
$$r_{\underline{v}_{\underline{w}}}: \bigoplus_{\tilde{w}\in W(1),\pi(\tilde{w})=w} RT_{\tilde{w}}\longrightarrow\bigoplus_{\tilde{v}\in W(1),\pi(\tilde{v})=v} RT_{\tilde{v}}$$
is defined term by term and extended linearly, where
$$r_{\underline{v}_{\underline{w}}}(T_{\tilde{w}})=T_t T_{\tilde{s}_{i_1}}^{e_1}(-c_{\tilde{s}_{i_1}})^{1-e_1}T_{\tilde{s}_{i_2}}^{e_2}(-c_{\tilde{s}_{i_2}})^{1-e_2}\dotsm T_{\tilde{s}_{i_n}}^{e_n}(-c_{\tilde{s}_{i_n}})^{1-e_n}T_{\tilde{\tau}}.$$

In other words, we fix $T_{\tilde{s}_{i_k}}$'s for $e_k=1$, and replace all the other $T_{\tilde{s}_{i_k}}$'s with $-c_{\tilde{s}_{i_k}}$'s. It is easy to see that $r_{\underline{v}_{\underline{w}}}$ is independent of choice of liftings.

\begin{example}
In $SL_3$ case, $W$ is generated by three elements $s_0,s_1,s_2$ with relations $s_i^2=1$ for all $i$ and $s_is_js_i=s_js_is_j$ if $i\neq j$. Let $\tilde{s}_0,\tilde{s}_1,\tilde{s}_2$ be liftings of $s_0,s_1,s_2$ respectively. Let $\underline{w}=s_0s_1s_2s_0s_1s_2$, $\underline{\tilde{w}}=t\tilde{s}_0\tilde{s}_1\tilde{s}_2\tilde{s}_0\tilde{s}_1\tilde{s}_2$ for some $t\in Z$. Let $(e_1,e_2,e_3,e_4,e_5,e_6)=(1,1,1,0,1,0)$ so that $\underline{v}_{\underline{w}}=s_0s_1s_21s_11$. Then $$r_{\underline{v}_{\underline{w}}}(T_{\tilde{w}})=T_tT_{\tilde{s}_0}T_{\tilde{s}_1}T_{\tilde{s}_2}(-c_{\tilde{s}_0})T_{\tilde{s}_1}(-c_{\tilde{s}_2})=T_tT_{\tilde{s}_0}T_{\tilde{s}_1}T_{\tilde{s}_2}c_{\tilde{s}_0}T_{\tilde{s}_1}c_{\tilde{s}_2}.$$
\end{example}

A priori, $r_{\underline{v}_{\underline{w}}}$ depends not only on the choice of reduced expression $\underline{w}$ but also on the choice of non-decreasing subexpression $\underline{v}_{\underline{w}}$. In the following part, we will show that, in fact, $r_{\underline{v}_{\underline{w}}}$ is independent of these choices, so the notation $r_{v,w}$ makes sense.

\begin{lem}
Let $w\in W$ with $\ell(w)=n$ and let $\underline{w}=s_{i_1}s_{i_2}\dotsm s_{i_n}\tau$ be a reduced expression of $w$. Let $\tilde{w}\in W(1)$ be a lifting of $w$ with $\underline{\tilde{w}}=t\tilde{s}_{i_1}\tilde{s}_{i_2}\dotsm\tilde{s}_{i_n}\tilde{\tau}$ for some $t\in Z$. Let $v\leq w$, and let $\underline{v}_{\underline{w}}=s_{i_1}^{e_1}s_{i_2}^{e_2}\dotsm s_{i_n}^{e_n}\tau$ and $\underline{v}'_{\underline{w}}=s_{i_1}^{f_1},s_{i_2}^{f_2}\dotsm s_{i_n}^{f_n}\tau$ be two non-decreasing subexpressions of $\underline{w}$ which both equal $v$. Then $r_{\underline{v}_{\underline{w}}}(T_{\tilde{w}})=r_{\underline{v}'_{\underline{w}}}(T_{\tilde{w}})$.
\end{lem}

\begin{proof}
We show this by induction on $l=\ell(w)+\ell(v)$.

If $l=0$, then $r_{\underline{v}_{\underline{w}}}(T_{\tilde{w}})=r_{\underline{v}'_{\underline{w}}}(T_{\tilde{w}})=T_{t\tilde{\tau}}$.

If $l=1$, then $\ell(w)=1$ and $\ell(v)=0$, so $r_{\underline{v}_{\underline{w}}}(T_{\tilde{w}})=r_{\underline{v}'_{\underline{w}}}(T_{\tilde{w}})=T_t(-c_{\tilde{s}_{i_1}})T_{\tilde{\tau}}$.

Now suppose that the statement is correct for $l<k$, and we consider the case when $l=k$.

\begin{itemize}
\item
If $e_1=f_1$, then by induction, the statement is correct.
\item
If $e_1\neq f_1$, then without loss of generality, we may assume that $e_1=1,f_1=0$, then
\begin{align}\nonumber
r_{\underline{v}_{\underline{w}}}(T_{\tilde{w}})
&=T_tT_{\tilde{s}_{i_1}}T_{\tilde{s}_{i_2}}^{e_2}(-c_{\tilde{s}_{i_2}})^{1-e_2}\dotsm T_{\tilde{s}_{i_n}}^{e_n}(-c_{\tilde{s}_{i_n}})^{1-e_n}T_{\tilde{\tau}},\\
r_{\underline{v}'_{\underline{w}}}(T_{\tilde{w}})
&=T_t(-c_{\tilde{s}_{i_1}})T_{\tilde{s}_{i_2}}^{f_2}(-c_{\tilde{s}_{i_2}})^{1-f_2}\dotsm T_{\tilde{s}_{i_n}}^{f_n}(-c_{\tilde{s}_{i_n}})^{1-f_n}T_{\tilde{\tau}}.\nonumber
\end{align}

Let $\ell(v)=m$, then we may assume that $f_{i_j}=1$ for $j\in\{j_1,\dotsm,j_m\}\subseteq\{2,\dotsm,n\}$ and $f_{i_j}=0$ otherwise in subexpression $\underline{v}'_{\underline{w}}$. But $s_{i_1}v<v$, so by strong exchange condition, $s_{i_1}v=s_{i_{j_1}}\dotsm\widehat{s_{i_{j_d}}}\dotsm s_{i_{j_m}}$ for some $j_d$. Then by induction,
$$r_{\underline{v}_{\underline{w}}}(T_{\tilde{w}})=T_tT_{\tilde{s}_{i_1}}T_{\tilde{s}_{i_2}}^{e'_2}(-c_{\tilde{s}_{i_2}})^{1-e'_2}\dotsm T_{\tilde{s}_{i_n}}^{e'_n}(-c_{\tilde{s}_{i_n}})^{1-e'_n}T_{\tilde{\tau}},$$
where $e'_{i_j}=1$ for $j\in\{j_1,\dotsm,\widehat{j_d},\dotsm,j_m\}$ and $e'_{i_j}=0$ otherwise.

Now the only difference between $r_{\underline{v}_{\underline{w}}}(T_{\tilde{w}})$ and $r_{\underline{v}'_{\underline{w}}}(T_{\tilde{w}})$ is that $r_{\underline{v}_{\underline{w}}}(T_{\tilde{w}})$ has $T_{\tilde{s}_{i_1}}$ and $-c_{\tilde{s}_{i_{j_d}}}$ as factors in the first and $j_d$th position respectively, while $r_{\underline{v}'_{\underline{w}}}(T_{\tilde{w}})$ has $-c_{\tilde{s}_{i_1}}$ and $T_{\tilde{s}_{i_{j_d}}}$ as factors in the first and $j_d$th position respectively. Factors in all other positions are the same for $r_{\underline{v}_{\underline{w}}}(T_{\tilde{w}})$ and $r_{\underline{v}'_{\underline{w}}}(T_{\tilde{w}})$.

Since $c_{\tilde{s}}$ is just a $R$-linear combination of elements in $Z$, it suffices to show that
$$\tilde{s}_{i_1}t_1\tilde{s}_{i_{j_1}}t_2\tilde{s}_{i_{j_2}}\dotsm t_{j_d}c_{\tilde{s}_{i_{j_d}}}\dotsm t_m\tilde{s}_{i_{j_m}}=c_{\tilde{s}_{i_1}}t_1\tilde{s}_{i_{j_1}}t_2\tilde{s}_{i_{j_2}}\dotsm t_{j_d}\tilde{s}_{i_{j_d}}\dotsm t_m\tilde{s}_{i_{j_m}}$$
for any $m$-tuple $(t_1,\dotsm,t_m)\in Z^m$, which holds by Lemma \ref{1}.
\end{itemize}

This finishes the proof.
\end{proof}

This lemma tells us that $r_{\underline{v}_{\underline{w}}}$ is independent of the choice of non-decreasing subexpression $\underline{v}_{\underline{w}}$. So we can rewrite the operator as $r_{v,\underline{w}}$.

\begin{thm}
Let $w\in W$ with $\ell(w)=n$ and let $\underline{w}_1=s_{11}s_{12}\dotsm s_{1n}\tau$ and $\underline{w}_2=s_{21}s_{22}\dotsm s_{2n}\tau$ be two reduced expressions of $w$. Let $\tilde{w}\in W(1)$ be a lifting of $w$ with two corresponding expressions $\underline{\tilde{w}}_1=t_1\tilde{s}_{11}\tilde{s}_{12}\dotsm \tilde{s}_{1n}\tilde{\tau}$ and $\underline{\tilde{w}}_2=t_2\tilde{s}_{21}\tilde{s}_{22}\dotsm\tilde{s}_{2n}\tilde{\tau}$ for some $t_1,t_2\in Z$ respectively. Let $v\leq w$ with $\ell(v)=m$, then $r_{v,\underline{w}_1}(T_{\tilde{w}})=r_{v,\underline{w}_2}(T_{\tilde{w}})$.
\end{thm}

\begin{proof}
Since $\underline{w}_1$ and $\underline{w}_2$ are two reduced expressions of $w$, then by Theorem $1.9$ in \cite{gl03} there exists a sequence
$$\underline{w}_1=(\underline{w})_1,(\underline{w})_2,...,(\underline{w})_d=\underline{w}_2$$
of reduced expressions of $w$ such that $(\underline{w})_i$ and $(\underline{w})_{i+1}$ differ only by a braid relation. So without loss of generality, we may assume that $\underline{w}_1$ and $\underline{w}_2$ differ only by a braid relation, and even more we may assume $n,m$ are both even and other cases for $n,m$ follow by similar proofs. Then
\begin{align}\nonumber
\underline{\tilde{w}}_1
&=\underbrace{\tilde{s}_{\alpha}\tilde{s}_{\beta}\dotsm\tilde{s}_{\alpha}\tilde{s}_{\beta}}_{n},\\\nonumber
\underline{\tilde{w}}_2
&=t\underbrace{\tilde{s}_{\beta}\tilde{s}_{\alpha}\dotsm\tilde{s}_{\beta}\tilde{s}_{\alpha}}_{n},\\\nonumber
v&=\underbrace{s_{\alpha}s_{\beta}\dotsm s_{\alpha}s_{\beta}}_m.\nonumber
\end{align}
for some $t\in Z$. Therefore,
\begin{align}\nonumber
r_{v,\underline{w}_1}(T_{\tilde{w}})
&=\underbrace{T_{\tilde{s}_{\alpha}}\dotsm T_{\tilde{s}_{\beta}}}_{m}\underbrace{(-c_{\tilde{s}_{\alpha}})\dotsm(-c_{\tilde{s}_{\beta}})}_{n-m}\\\nonumber
&=\underbrace{T_{\tilde{s}_{\alpha}}\dotsm T_{\tilde{s}_{\beta}}}_{m}\underbrace{c_{\tilde{s}_{\alpha}}\dotsm c_{\tilde{s}_{\beta}}}_{n-m},\\\nonumber
r_{v,\underline{w}_2}(T_{\tilde{w}})
&=T_t(-c_{\tilde{s}_{\beta}})\underbrace{T_{\tilde{s}_{\alpha}}\dotsm T_{\tilde{s}_{\beta}}}_{m}\underbrace{(-c_{\tilde{s}_{\alpha}})\dotsm(-c_{\tilde{s}_{\alpha}})}_{n-m-1}\\\nonumber
&=T_tc_{\tilde{s}_{\beta}}\underbrace{T_{\tilde{s}_{\alpha}}\dotsm T_{\tilde{s}_{\beta}}}_{m}\underbrace{c_{\tilde{s}_{\alpha}}\dotsm c_{\tilde{s}_{\alpha}}}_{n-m-1}.\nonumber
\end{align}

It is enough to show that $\underbrace{\tilde{s}_{\alpha}\dotsm\tilde{s}_{\beta}}_{m}\underbrace{c_{\tilde{s}_{\alpha}}\dotsm c_{\tilde{s}_{\beta}}}_{n-m}=tc_{\tilde{s}_{\beta}}\underbrace{\tilde{s}_{\alpha}\dotsm\tilde{s}_{\beta}}_{m}\underbrace{c_{\tilde{s}_{\alpha}}\dotsm c_{\tilde{s}_{\alpha}}}_{n-m-1}$. But $t\tilde{s}_{\beta}=\underbrace{\tilde{s}_{\alpha}\dotsm\tilde{s}_{\alpha}}_{n-1}\tilde{s}_{\beta}\underbrace{\tilde{s}_{\alpha}^{-1}\dotsm\tilde{s}_{\alpha}^{-1}}_{n-1}$, so $tc_{\tilde{s}_{\beta}}=\underbrace{\tilde{s}_{\alpha}\dotsm\tilde{s}_{\alpha}}_{n-1}c_{\tilde{s}_{\beta}}\underbrace{\tilde{s}_{\alpha}^{-1}\dotsm\tilde{s}_{\alpha}^{-1}}_{n-1}$. Therefore
\begin{align}\nonumber
tc_{\tilde{s}_{\beta}}\underbrace{\tilde{s}_{\alpha}\dotsm\tilde{s}_{\beta}}_{m}\underbrace{c_{\tilde{s}_{\alpha}}\dotsm c_{\tilde{s}_{\alpha}}}_{n-m-1}
&=\underbrace{\tilde{s}_{\alpha}\dotsm\tilde{s}_{\alpha}}_{n-1}c_{\tilde{s}_{\beta}}\underbrace{\tilde{s}_{\alpha}^{-1}\dotsm\tilde{s}_{\alpha}^{-1}}_{n-1}\underbrace{\tilde{s}_{\alpha}\dotsm\tilde{s}_{\beta}}_{m}\underbrace{c_{\tilde{s}_{\alpha}}\dotsm c_{\tilde{s}_{\alpha}}}_{n-m-1}\\\nonumber
&=\underbrace{\tilde{s}_{\alpha}\dotsm\tilde{s}_{\beta}}_{m}\underbrace{\tilde{s}_{\alpha}\dotsm\tilde{s}_{\alpha}}_{n-m-1}c_{\tilde{s}_{\beta}}\underbrace{\tilde{s}_{\alpha}^{-1}\dotsm\tilde{s}_{\alpha}^{-1}}_{n-m-1}\underbrace{c_{\tilde{s}_{\alpha}}\dotsm c_{\tilde{s}_{\alpha}}}_{n-m-1}\\\nonumber
&=\underbrace{\tilde{s}_{\alpha}\dotsm\tilde{s}_{\beta}}_{m}c_{\tilde{s}_{\alpha}}\underbrace{\tilde{s}_{\beta}\dotsm\tilde{s}_{\alpha}}_{n-m-2}c_{\tilde{s}_{\beta}}\underbrace{\tilde{s}_{\alpha}^{-1}\dotsm\tilde{s}_{\beta}^{-1}}_{n-m-2}\underbrace{c_{\tilde{s}_{\beta}}\dotsm c_{\tilde{s}_{\alpha}}}_{n-m-2}\\\nonumber
&=\underbrace{\tilde{s}_{\alpha}\dotsm\tilde{s}_{\beta}}_{m}c_{\tilde{s}_{\alpha}}c_{\tilde{s}_{\beta}}\underbrace{\tilde{s}_{\alpha}\dotsm\tilde{s}_{\alpha}}_{n-m-3}c_{\tilde{s}_{\beta}}\underbrace{\tilde{s}_{\alpha}^{-1}\dotsm\tilde{s}_{\alpha}^{-1}}_{n-m-3}\underbrace{c_{\tilde{s}_{\alpha}}\dotsm c_{\tilde{s}_{\alpha}}}_{n-m-3}\\\nonumber
&\dotsm\\\nonumber
&=\underbrace{\tilde{s}_{\alpha}\dotsm\tilde{s}_{\beta}}_{m}\underbrace{c_{\tilde{s}_{\alpha}}\dotsm c_{\tilde{s}_{\beta}}}_{n-m}.
\end{align}

The third equality holds since $$\underbrace{\tilde{s}_{\alpha}\dotsm\tilde{s}_{\alpha}}_{n-m-1}c_{\tilde{s}_{\beta}}\underbrace{\tilde{s}_{\alpha}^{-1}\dotsm\tilde{s}_{\alpha}^{-1}}_{n-m-1}c_{\tilde{s}_{\alpha}}=c_{\tilde{s}_{\alpha}}\underbrace{\tilde{s}_{\beta}\dotsm\tilde{s}_{\alpha}}_{n-m-2}c_{\tilde{s}_{\beta}}\\\underbrace{\tilde{s}_{\alpha}^{-1}\dotsm\tilde{s}_{\beta}^{-1}}_{n-m-2}$$ which is true becasue $\underbrace{\tilde{s}_{\beta}\dotsm\tilde{s}_{\alpha}}_{n-m-2}c_{\tilde{s}_{\beta}}\underbrace{\tilde{s}_{\alpha}^{-1}\dotsm\tilde{s}_{\beta}^{-1}}_{n-m-2}\in R[Z]$ and $\tilde{s}_{\alpha}t'\tilde{s}_{\alpha}^{-1}c_{\tilde{s}_{\alpha}}=c_{\tilde{s}_{\alpha}t'\tilde{s}_{\alpha}^{-1}\tilde{s}_{\alpha}}=c_{\tilde{s}_{\alpha}t'}=c_{\tilde{s}_{\alpha}}t'$ for any $t'\in Z$. And all subsequent equalities hold for a similar reason.
\end{proof}

As the main result of this section, this theorem guarantees that $r_{v,\underline{w}}$ is independent of the choice of reduced expression of $w$. So we can rewrite the operator as $r_{v,w}$, which is what we need and will be used later.

By definition of the operator, we can easily get the following propositions.

\begin{prop}
Let $u,v,w\in W$ and suppose $u\leq v\leq w$, then
$$r_{u,v}r_{v,w}=r_{u,w}.$$
\end{prop}

\begin{prop}\label{2}
Let $u,v,w\in W$ and $\tilde{u},\tilde{w}\in W(1)$ be liftings of $u,w$ respectively. 
\begin{itemize}
\item[(1)]
If $v\leq w$ and $\ell(uv)=\ell(u)+\ell(v), \ell(uw)=\ell(u)+\ell(w)$, then
$$T_{\tilde{u}}r_{v,w}(T_{\tilde{w}})=r_{uv,uw}(T_{\tilde{u}\tilde{w}}).$$
\item[(2)]
If $v\leq w$ and $\ell(vu)=\ell(v)+\ell(u), \ell(wu)=\ell(w)+\ell(u)$, then $$r_{v,w}(T_{\tilde{w}})T_{\tilde{u}}=r_{vu,wu}(T_{\tilde{w}\tilde{u}}).$$
\end{itemize}
\end{prop}

\section{Maximal Length Elements}

Let $\mathcal{Z}_R(0,c_{\tilde{s}})$ be the center of $\mathcal{H}_R(0,c_{\tilde{s}})$ and $h\in \mathcal{Z}_R(0,c_{\tilde{s}})$. Then 
$$h=\sum_{\tilde{w}\in W(1)} a_{\tilde{w}}T_{\tilde{w}},\quad\text{for some}\ \ a_{\tilde{w}}\in R.$$
Set $\text{supp}(h)=\{\tilde{w}\in W(1)| a_{\tilde{w}}\neq 0\}$. Let $\text{supp}(h)_{\max}$ be the set of maximal length elements in $\text{supp}(h)$. The following theorem tells what $\text{supp}(h)_{\max}$ is comprised of.

\begin{thm}\label{11}
Suppose $h\in\mathcal{Z}_R(0,c_{\tilde{s}})$, then $\text{supp}(h)_{\max}$ is a finite union of finite conjugacy classes in $W(1)$.
\end{thm}

This theorem comes from the following results.

\begin{lem}\label{3}
Let $\tilde{s}\in S^{\text{aff}}(1)$, $h\in \mathcal{Z}_{R}(0,c_{\tilde{s}})$ and $\tilde{w}\in \text{supp}(h)_{\max}$. If $\ell(\tilde{s}\tilde{w})>\ell(\tilde{w})$ or $\ell(\tilde{w}\tilde{s})>\ell(\tilde{w})$, then
$\tilde{s}\tilde{w}\tilde{s}^{-1}\in\text{supp}(h)_{\max}$ and $a_{\tilde{s}\tilde{w}\tilde{s}^{-1}}=a_{\tilde{w}}$.
\end{lem}

\begin{proof}
Without loss of generality, we may assume that $\ell(\tilde{s}\tilde{w})>\ell(\tilde{w})$. Then $\tilde{s}\tilde{w}\in \text{supp}(T_{\tilde{s}} h)=\text{supp}(h T_{\tilde{s}})$, and
$$\text{supp}(T_{\tilde{s}} h)_{\max}=\{\tilde{s}\tilde{x}|\tilde{x}\in\text{supp}(h)_{\max},\ell(\tilde{s}\tilde{x})>\ell(\tilde{x})\},$$
$$\text{supp}(h T_{\tilde{s}})_{\max}=\{\tilde{y}\tilde{s}|\tilde{y}\in\text{supp}(h)_{\max},\ell(\tilde{y}\tilde{s})>\ell(\tilde{y})\}.$$
Both sets are nonempty because $\tilde{s}\tilde{w}\in \text{supp}(T_{\tilde{s}}h)_{\max}$. Therefore, $\tilde{s}\tilde{w}\tilde{s}^{-1}\in\text{supp}(h)_{\max}$ and $\ell(\tilde{s}\tilde{w}\tilde{s}^{-1})=\ell(\tilde{w})$. The $R$-coefficient of $T_{\tilde{s}\tilde{w}}$ in $T_{\tilde{s}} h$ is $a_{\tilde{w}}$ and the $R$-coefficient of $T_{\tilde{s}\tilde{w}}$ in $h T_{\tilde{s}}$ is $a_{\tilde{s}\tilde{w}\tilde{s}^{-1}}$. Thus $a_{\tilde{s}\tilde{w}\tilde{s}^{-1}}=a_{\tilde{w}}$.
\end{proof}

We recall Main Theorem in \cite{sr}:

\begin{thm}\label{12}
Fix $w\in W$. If $w\notin\Lambda$ then there exists $s\in S^{\text{aff}}$ and $s_1,\dotsm,s_n\in S^{\text{aff}}$ such that, setting $w'\stackrel{\text{def}}=s_n\dotsm s_1 w s_1\dotsm s_n$,
\begin{itemize}
\item
$\ell(s_i\dotsm s_1 w s_1\dotsm s_i)=\ell(w)$ for all $i$,
\item
$\ell(sw's)>\ell(w')$.
\end{itemize}
\end{thm}

\begin{lem}\label{4}
Suppose $h\in \mathcal{Z}_{R}(0,c_{\tilde{s}})$ and $\tilde{w}\in\text{supp}(h)_{\max}$, then $\tilde{w}\in \Lambda(1)$.
\end{lem}

\begin{proof}
We prove by contradiction. Assume $\tilde{w}\in\text{supp}(h)_{\max}$ but $\tilde{w}\notin \Lambda(1)$.

By Theorem \ref{12}, $\exists \tilde{s}\in S^{\text{aff}}(1)$ and $\tilde{s}_1,\tilde{s}_2,\dotsm,\tilde{s}_n\in S^{\text{aff}}(1)$ such that
\begin{itemize}
\item
$\ell(\tilde{s}_i\dotsm\tilde{s}_1\tilde{w}\tilde{s}_1^{-1}\dotsm\tilde{s}_i^{-1})=\ell(\tilde{w})$ for all $i$,
\item
$\pi(\tilde{s}_i\tilde{s}_{i-1}\dotsm \tilde{s}_1\tilde{w}\tilde{s}_1^{-1}\dotsm \tilde{s}_{i-1}^{-1}\tilde{s}_i^{-1})\neq \pi(\tilde{s}_{i-1}\dotsm \tilde{s}_1\tilde{w}\tilde{s}_1^{-1}\dotsm \tilde{s}_{i-1}^{-1})$ for all $i$,
\item
$\ell(\tilde{s}\tilde{w'}\tilde{s}^{-1})>\ell(\tilde{w'})$, where $\tilde{w'}=\tilde{s}_n\dotsm \tilde{s}_2\tilde{s}_1\tilde{w}\tilde{s}_1^{-1}\tilde{s}_2^{-1}\dotsm \tilde{s}_n^{-1}$.
\end{itemize}

By Lemma \ref{3}, $\tilde{s}_i\dotsm \tilde{s}_1\tilde{w}\tilde{s}_1^{-1}\dotsm\tilde{s}_i^{-1}\in \text{supp}(h)_{\max}$ for all $i$, in particular, $\tilde{w'}=\tilde{s}_n\dotsm\tilde{s}_2\tilde{s}_1\tilde{w}\tilde{s}_1^{-1}\tilde{s}_2^{-1}\dotsm\tilde{s}_n^{-1}\in \text{supp}(h)_{\max}$.

By Lemma \ref{3} again, $\tilde{s}\tilde{w'}\tilde{s}^{-1}\in \text{supp}(h)_{\max}$. But $\ell(\tilde{s}\tilde{w'}\tilde{s}^{-1})>\ell(\tilde{w'})$, which is a contradiction.
\end{proof}

\begin{proof}[Proof of Theorem \ref{11}]
It suffices to show that if $h\in \mathcal{Z}_{R}(0,c_{\tilde{s}})$, $\tilde{w}\in\text{supp}(h)_{\max}$ and $Cl(\tilde{w})$ is the $W(1)$-conjugacy class of $\tilde{w}$ in $W(1)$, then $Cl(\tilde{w})\subseteq\text{supp}(h)_{\max}$ and $a_{\tilde{w'}}=a_{\tilde{w}}$ for any $\tilde{w'}\in Cl(\tilde{w})$.

By Lemma \ref{3} and Lemma \ref{4}, $\tilde{x}\tilde{w}\tilde{x}^{-1}\in\text{supp}(h)_{\max}$ and $a_{\tilde{x}\tilde{w}\tilde{x}^{-1}}=a_{\tilde{w}}$ for any $\tilde{x}\in W^{\text{aff}}(1)$. It remains to show that $\tilde{\tau} \tilde{w}\tilde{\tau}^{-1}\in\text{supp}(h)_{\max}$ and $a_{\tilde{\tau} \tilde{w}\tilde{\tau}^{-1}}=a_{\tilde{w}}$ for any $\tilde{\tau}\in\Omega(1)$. But $\tilde{\tau}\tilde{w}\in \text{supp}(T_{\tilde{\tau}} h)=\text{supp}(h T_{\tilde{\tau}})$, and
$$\text{supp}(T_{\tilde{\tau}} h)_{\max}=\{\tilde{\tau}\tilde{x}|\tilde{x}\in\text{supp}(h)_{\max}\},$$
$$\text{supp}(h T_{\tilde{\tau}})_{\max}=\{\tilde{y}\tilde{\tau}|\tilde{y}\in\text{supp}(h)_{\max}\}.$$
Both sets are nonempty because $\tilde{\tau}\tilde{w}\in \text{supp}(T_{\tilde{\tau}}h)_{\max}$. Therefore, $\tilde{\tau}\tilde{w}\tilde{\tau}^{-1}\in\text{supp}(h)_{\max}$. The $R$-coefficient of $T_{\tilde{\tau}\tilde{w}}$ in $T_{\tilde{\tau}} h$ is $a_{\tilde{w}}$ and the $R$-coefficient of $T_{\tilde{\tau}\tilde{w}}$ in $h T_{\tilde{\tau}}$ is $a_{\tilde{\tau}\tilde{w}\tilde{\tau}^{-1}}$. Thus $a_{\tilde{\tau}\tilde{w}\tilde{\tau}^{-1}}=a_{\tilde{w}}$.
\end{proof}

It is well-known that a conjugacy class $C$ of $W$ is finite if and only if $C$ is contained in $\Lambda$. The same is true for $W(1)$ and $\Lambda(1)$. So $\text{supp}(h)$ is a finite union of some conjugacy classes in $\Lambda(1)$.

In fact, Theorem $4.1$ is valid even without the restriction on $q_{\tilde{s}}$, because the proof does not use the fact that $q_{\tilde{s}}=0$. 

\section{Some Technical Results}

Let $C$ be a finite conjugacy class in $W(1)$. Set 
$$h_{\lambda,C}=\sum_{\tilde{\lambda}\in \pi^{-1}(\lambda)\cap C}T_{\tilde{\lambda}},$$
for every $\lambda\in\pi(C)$. We may just write $h_{\lambda}$ for $h_{\lambda,C}$ when there is no ambiguity. Now we prove some properties of $r_{x,\lambda}(h_{\lambda})$.

\begin{lem}\label{6}
Let $C$ be a finite conjugacy class in $W(1)$, $\lambda\in\pi(C)$ and $s\in S^{\text{aff}}$. Let $x\in W$ with $x<sx$ or $x<xs$. Suppose that $x\leq \lambda$ and $x\leq s\lambda s$, then
$$r_{x,\lambda}(h_{\lambda})=r_{x,s\lambda s}(h_{s\lambda s}).$$
\end{lem}

\begin{proof}
Without loss of generality, we may assume $x<sx$.

If $s\lambda s=\lambda$, then it is clearly true.

If $s\lambda s\neq\lambda$, then without loss of generality, we may assume $s\lambda<\lambda$. In this case, $x\leq s\lambda$ by Lemma \ref{14}. Thus
$$r_{x,\lambda}(h_{\lambda})=r_{x,s\lambda}(r_{s\lambda,\lambda}(h_{\lambda})),\quad r_{x,s\lambda s}(h_{s\lambda s})=r_{x,s\lambda}(r_{s\lambda,s\lambda s}(h_{s\lambda s})).$$
It suffices to show that $r_{s\lambda,\lambda}(h_{\lambda})=r_{s\lambda,s\lambda s}(h_{s\lambda s})$. 

Since $c_{\tilde{s}^{-1}}\in R[Z]$, we may assume that
$$c_{\tilde{s}^{-1}}=\sum_{t\in Z}b_t t,\quad\text{for some}\ \ b_t\in R.$$
Then
\begin{align}\nonumber
&r_{s\lambda,\lambda}(h_{\lambda})=r_{s\lambda,s\lambda s}(h_{s\lambda s})\\\nonumber
\Longleftrightarrow&\sum_{\tilde{\lambda}\in\pi^{-1}(\lambda)\cap C}-c_{\tilde{s}^{-1}}T_{\tilde{s}\tilde{\lambda}}=\sum_{\tilde{\lambda}\in\pi^{-1}(\lambda)\cap C}-T_{\tilde{s}\tilde{\lambda}}c_{\tilde{s}^{-1}}\\\nonumber
\Longleftrightarrow&\sum_{\tilde{\lambda}\in\pi^{-1}(\lambda)\cap C}-c_{\tilde{s}^{-1}}T_{\tilde{s}\tilde{\lambda}}=\sum_{\tilde{\lambda}\in\pi^{-1}(\lambda)\cap C}-T_{\tilde{s}\tilde{\lambda}}(\tilde{s}c_{\tilde{s}^{-1}}\tilde{s}^{-1})\\\nonumber
\Longleftrightarrow&\sum_{\tilde{\lambda}\in\pi^{-1}(\lambda)\cap C}-(\sum_{t\in Z}b_t t)T_{\tilde{s}\tilde{\lambda}}=\sum_{\tilde{\lambda}\in\pi^{-1}(\lambda)\cap C}-T_{\tilde{s}\tilde{\lambda}}(\tilde{s}(\sum_{t\in Z}b_t t)\tilde{s}^{-1})\\\nonumber
\Longleftrightarrow&\sum_{t\in Z}b_t\sum_{\tilde{\lambda}\in\pi^{-1}(\lambda)\cap C}-T_{t\tilde{s}\tilde{\lambda}}=\sum_{t\in Z}b_t\sum_{\tilde{\lambda}\in\pi^{-1}(\lambda)\cap C}-T_{\tilde{s}\tilde{\lambda}(\tilde{s}t\tilde{s}^{-1})}.\nonumber
\end{align}
We want to show the last equation. It suffices to show that
\begin{align}\nonumber
&\sum_{\tilde{\lambda}\in \pi^{-1}(\lambda)\cap C}t\tilde{s}\tilde{\lambda}=\sum_{\tilde{\lambda}\in \pi^{-1}(\lambda)\cap C}\tilde{s}\tilde{\lambda}(\tilde{s}t\tilde{s}^{-1})\\\nonumber
\Longleftrightarrow&\sum_{\tilde{\lambda}\in \pi^{-1}(\lambda)\cap C}\tilde{s}^{-1}t\tilde{s}\tilde{\lambda}=\sum_{\tilde{\lambda}\in \pi^{-1}(\lambda)\cap C}\tilde{\lambda}(\tilde{s}t\tilde{s}^{-1})\\\nonumber
\Longleftrightarrow&(\tilde{s}^{-1}t\tilde{s})(\sum_{\tilde{\lambda}\in \pi^{-1}(\lambda)\cap C}\tilde{\lambda})(\tilde{s}^{-1}t\tilde{s})^{-1}=\sum_{\tilde{\lambda}\in \pi^{-1}(\lambda)\cap C}\tilde{\lambda}\nonumber
\end{align}
for any $t\in Z$. The last equation holds because $\sum_{\tilde{\lambda}\in \pi^{-1}(\lambda)\cap C}\tilde{\lambda}$ is fixed by $Z$.
\end{proof}

For $w,w'\in W$, we write $w\stackrel{s}\rightarrow w'$ if $w'=sws$ and $\ell(w')=\ell(w)-2$.

For $w,w'\in W$, we write $w\stackrel{s}\sim w'$ if $w'=sws$, $\ell(w')=\ell(w)$, and $sw>w$ or $ws>w$. We write $w\sim w'$ if $\exists$ a sequence $$w=w_0,w_1,...,w_n=w'$$ such that $w_{i-1}\stackrel{s_i}{\sim} w_i$ for every $i$ and some $s_i\in S^{\text{aff}}$. If $\lambda,\lambda'$ are in the same finite conjugacy class in $W$, then $\lambda'=w\lambda w^{-1}$ for some $w\in W$. Since $W=\Lambda\rtimes W_0$, we can write $w=w_0\lambda''$ for some $w_0\in W_0$ and $\lambda''\in\Lambda$. Thus by commutativity of $\Lambda$, $\lambda'=(w_0\lambda'')\lambda(w_0\lambda'')^{-1}=w_0\lambda w_0^{-1}$. Therefore, $\lambda\sim\lambda'.$

\begin{lem}\label{7}
Let $C$ be a finite conjugacy class in $W(1)$ and $\lambda\in\pi(C)$. Let $x,x'\in W$ and $x\leq\lambda$. Suppose $x\sim x'$ is given by 
$$x=x_0\stackrel{s_1}\sim x_1\stackrel{s_2}\sim\dotsm\stackrel{s_n}\sim x_n=x',$$ for some $s_i\in S^{\text{aff}}$. Let $w=s_n\dotsm s_1$. Then $\exists \lambda'\sim\lambda, \tilde{w}\in W(1)$ a lifting of $w$, such that $x'\leq\lambda'$ and 
$$\tilde{w}\bullet(r_{x,\lambda}(h_{\lambda}))=r_{x',\lambda'}(h_{\lambda'}).$$
\end{lem}

\begin{proof}
It suffices to prove the case where $x\stackrel{s} \sim x'$ for some $s\in S^{\text{aff}}$, i.e. $x'=sxs$. Without loss of generality, we may assume that $sx>x$.
\begin{itemize}
\item
If $s\lambda>\lambda$, then by Lemma \ref{14} $sxs\leq s\lambda s$. It is enough to show that $T_{\tilde{s}}r_{x,\lambda}(h_{\lambda})=r_{sxs,s\lambda s}(h_{s\lambda s})T_{\tilde{s}}$ for any $\tilde{s}\in S^{aff}(1)$ with $\pi(\tilde{s})=s$. But
\begin{align}\nonumber
T_{\tilde{s}}r_{x,\lambda}(h_{\lambda})
&=r_{sx,s\lambda}(T_{\tilde{s}}h_{\lambda})\\\nonumber
&=r_{sx,s\lambda}(h_{s\lambda s}T_{\tilde{s}})\\\nonumber
&=r_{sxs,s\lambda s}(h_{s\lambda s})T_{\tilde{s}}.\nonumber
\end{align}
The second the equality holds because $\tilde{s}\bullet h_{\lambda}=h_{s\lambda s}$, and the other equalities hold by Proposition \ref{2}. Therefore,
$$\tilde{s}\bullet(r_{x,\lambda}(h_{\lambda}))=r_{x',s\lambda s}(h_{s\lambda s}).$$
\item
If $s\lambda<\lambda$, then by Lemma \ref{14} $sxs<sx\leq \lambda$ and $x\leq s\lambda< s\lambda s$. Therefore, for any $\tilde{s}\in S^{aff}(1)$ with $\pi(\tilde{s})=s$, we have
\begin{align}\nonumber
T_{\tilde{s}}(r_{x,\lambda}(h_{\lambda}))
&=T_{\tilde{s}}(r_{x,s\lambda s}(h_{s\lambda s}))\\\nonumber
&=(r_{sxs,\lambda}(h_{\lambda})).\nonumber
\end{align}
The first equality holds by Lemma \ref{6} and the second equality holds by a similar proof as the last step. Thus
$$\tilde{s}\bullet(r_{x,\lambda}(h_{\lambda}))=r_{x',\lambda}(h_{\lambda}).$$
\end{itemize}

This finishes the proof.
\end{proof}

Recall that $\nu$ is the homomorphism which defines the action of $\Lambda$. Set $\Lambda^+=\{\lambda\in\Lambda| \beta(\nu(\lambda))\geq 0, \forall \beta\in\Sigma^+\}$ where $\Sigma^+$ is the set of positive roots in $\Sigma$. A element in $\Lambda$ is called dominant if it is contained in $\Lambda^+$. Let $\mu_0\in\Lambda^+$ and $\lambda\in\Lambda$. Let $\lambda_0$ be the unique dominant element in $\{\lambda'\in\Lambda| \lambda'\sim\lambda\}$. Suppose $\mu_0\leq\lambda$, then by Theorem $1.3$ in \cite{hh} and Proposition $6.1$ in \cite{th}, $\mu_0\leq\lambda_0$. We have the following result.

\begin{lem}\label{8}
Let $\mu_0\in\Lambda^+ and \lambda\in\Lambda$. Let $\lambda_0$ be the unique dominant element in $\{\lambda'\in\Lambda| \lambda'\sim\lambda\}$. Suppose $\mu_0\leq\lambda$, then $\exists$ a sequence
$$\lambda_0,\lambda_1,\dotsm,\lambda_n=\lambda$$
such that $\lambda_{i-1}\stackrel{s_i}\sim\lambda_i$ for every $i$ and some $s_i\in S_0$, and $\mu_0\leq\lambda_i$ for all $i$.
\end{lem}

\begin{proof}
Since $\lambda\sim\lambda_0$, there exists $w\in S_0$ such that $\lambda=w\lambda_0 w^{-1}$. We prove the statement by induction on $l=\ell(w)$.

If $l=0,1$, then it is obvious.

Now suppose that the statement is correct for $l<k$, and we consider the case when $l=k$. Let $w=s_{i_k}\dotsm s_{i_1}$ and it suffices to show that $\mu_0\leq s_{i_k}\lambda s_{i_k}$.

If $s_{i_k}\lambda s_{i_k}=\lambda$, then it is obvious.

If $s_{i_k}\lambda s_{i_k}\neq \lambda$, then by $s_{i_k}w<w$ and Lemma \ref{13}, $w(\mathfrak{A}_0)$ and $\mathfrak{A}_0$ are on different sides of $H_{s_{i_k}}$. On the other hand, $s_{i_k}\lambda s_{i_k}\neq \lambda$, then $\nu(\lambda)=\nu(w\lambda_0 w^{-1})=w(\nu(\lambda_0))\in w(\bar{\mathfrak{C}}_0)\backslash H_{s_{i_k}}$. Thus $\lambda(\mathfrak{A}_0)=\mathfrak{A}_0+\nu(\lambda)$ and $\mathfrak{A}_0$ are on different sides of $H_{s_{i_k}}$, i.e. $s_{i_k}\lambda<\lambda$ by Lemma \ref{13}. We also have $s_{i_k}\mu_0>\mu_0$, thus by Lemma \ref{14} $\mu_0\leq s_{i_k}\lambda<s_{i_k}\lambda s_{i_k}$, which finishes the proof. 
\end{proof}

\begin{thm}\label{9}
Let $C$ be a finite conjugacy class in $W(1)$, $\lambda_1,\lambda_2\in \pi(C)$ and $x\in W$. Suppose $x\leq\lambda_1,\lambda_2$, then
$$r_{x,\lambda_1}(h_{\lambda_1})=r_{x,\lambda_2}(h_{\lambda_2}).$$
\end{thm}

\begin{proof}
We prove it by induction on $d=\ell(\lambda_1)-\ell(x)=\ell(\lambda_2)-\ell(x)$.

If $d=0$, then it is obvious since $x=\lambda_1=\lambda_2.$

Now suppose $d>0$.
\begin{itemize}
\item
If $x\notin \Lambda$, then by Theorem \ref{12} there exist $s_1,s_2,\dotsm,s_n,s'\in S^{\text{aff}}$ such that $s_is_{i-1}\dotsm s_1xs_1\dotsm s_{i-1}s_i\stackrel{s_{i+1}}\sim s_{i+1}s_is_{i-1}\dotsm s_1xs_1\dotsm s_{i-1}s_is_{i+1}$ for all $i$ and $s's_ns_{n-1}\dotsm s_1xs_1\dotsm s_{n-1}s_ns'\stackrel{s'}\rightarrow s_ns_{n-1}\dotsm s_1xs_1\dotsm s_{n-1}s_n$. Let $\tilde{w}\in W^{\text{aff}}(1)$ be a lifting of $s_ns_{n-1}\dotsm s_1$ and $x'=s_ns_{n-1}\dotsm s_1xs_1\dotsm s_{n-1}s_n$. Then by Lemma \ref{7},
$$\tilde{w}\bullet(r_{x,\lambda_1}(h_{\lambda_1}))=r_{x',\lambda_1'}(h_{\lambda_1'}),\quad \tilde{w}\bullet(r_{x,\lambda_2}(h_{\lambda_2}))=r_{x',\lambda_2'}(h_{\lambda_2'}),$$
for some $\lambda_1'\sim\lambda_1, \lambda_2'\sim\lambda_2$. We have $\lambda_1'\sim\lambda_2'$ because $\lambda_1\sim\lambda_2$.

It suffices to show that $r_{x',\lambda_1'}(h_{\lambda_1'})=r_{x',\lambda_2'}(h_{\lambda_2'}).$ It can be easily checked by Lemma \ref{14} that $s'x'\leq \lambda_j'$ or $s'\lambda_j's'$ for $j=1,2$. Without loss of generality, we may assume that $s'x'\leq \lambda_1',\lambda_2'$, then
$$r_{x',\lambda_1'}(h_{\lambda_1'})=r_{x',s'x'}(r_{s'x',\lambda_1'}(h_{\lambda_1'}))=r_{x',s'x'}(r_{s'x',\lambda_2'}(h_{\lambda_2'}))=r_{x',\lambda_2'}(h_{\lambda_2'}),$$
where the second equality holds by induction. If $s'x'\leq s'\lambda_j's'$, then $x'<s'\lambda_j's'$. By Lemma \ref{6}, $$r_{x',\lambda_j'}(h_{\lambda_j'})=r_{x',s'\lambda_j's'}(h_{s'\lambda_j's'})=r_{x',s'x'}(r_{s'x',s'\lambda_j's'}(h_{s'\lambda_j's'})),$$ and we can apply a similar proof as above.
\item
If $x\in \Lambda$, then there exist $w\in S_0$ such that $x_0=wxw^{-1}\in\Lambda^+$. Let $\tilde{w}\in W(1)$ be a lifting of $w$, then by Lemma \ref{7},
$$\tilde{w}\bullet(r_{x,\lambda_1}(h_{\lambda_1}))=r_{x_0,\lambda_1'}(h_{\lambda_1'}),\quad \tilde{w}\bullet(r_{x,\lambda_2}(h_{\lambda_2}))=r_{x_0,\lambda_2'}(h_{\lambda_2'}),$$
for some $\lambda_1'\sim\lambda_1,\lambda_2'\sim\lambda_2$. We have $\lambda_1'\sim\lambda_2'$ because $\lambda_1\sim\lambda_2$.

It suffices to show that $r_{x_0,\lambda_1'}(h_{\lambda_1'})=r_{x_0,\lambda_2'}(h_{\lambda_2'})$. By Lemma \ref{7} and \ref{8}, $r_{x_0,\lambda_1'}(h_{\lambda_1'})=r_{x_0,\lambda_0}(h_{\lambda_0})=r_{x_0,\lambda_2'}(h_{\lambda_2'})$ where $\lambda_0\in\Lambda^+$ and $\lambda_0\sim\lambda_1',\lambda_0\sim\lambda_2'$.
\end{itemize}

This finishes the proof.
\end{proof}

\section{Center of $\mathcal{H}_R(0,c_{\tilde{s}})$}

Let $C$ be a finite conjugacy class in $W(1)$. Then $C\subset\Lambda(1), \pi(C)\subset\Lambda$ and there is a unique element $\lambda_0\in\pi(C)\cap\Lambda^+$. Set $$\text{Adm}(C)=\text{Adm}(\lambda_0)=\{w\in W| w\leq\lambda\ \ \text{for some}\ \ \lambda\in\pi(C)\}.$$ We define
$$h_C=\sum_{w\in \text{Adm}(C)} h_w,$$
where $h_w=h_{w,C}$ if $w\in\pi(C)$ and otherwise $h_{w}=r_{w,\lambda}(h_{\lambda})$ for any $\lambda\in \pi(C)$ with $\lambda>w$. By Theorem \ref{9}, $h_C$ is well defined.

\begin{lem}
Suppose $C$ be a finite conjugacy class in $W(1)$. Then $h_C\in\mathcal{Z}_R(0,c_{\tilde{s}})$. 
\end{lem}

\begin{proof}
For any $\tilde{\tau}\in\Omega(1)$ with $\pi(\tilde{\tau})=\tau$,
\begin{align}\nonumber
T_{\tilde{\tau}}h_C
&=\sum_{w\in\text{Adm}(C)}T_{\tilde{\tau}}h_w\\\nonumber
&=\sum_{w\in\text{Adm}(C)}h_{\tau w\tau^{-1}}T_{\tilde{\tau}}\\\nonumber
&=(\tilde{\tau}\bullet(\sum_{w\in\text{Adm}(C)}h_w))T_{\tilde{\tau}}\\\nonumber
&=h_C T_{\tilde{\tau}}.\nonumber
\end{align}
The second equality holds by definition of $h_C$ and Proposition \ref{2}, and the third equality holds because $h_C$ is stable under the action of $W(1)$.

It remains to show that for any $\tilde{s}\in S^{\text{aff}}(1)$ with $\pi(\tilde{s})=s$, $T_{\tilde{s}}h_C=h_CT_{\tilde{s}}$. The left hand side
$$T_{\tilde{s}}h_C=\sum_{w\in\text{Adm}(C)}T_{\tilde{s}}h_w=\sum_{x,sx\in\text{Adm}(C)}T_{\tilde{s}}h_x+\sum_{y\in\text{Adm}(C),sy\notin\text{Adm}(C)}T_{\tilde{s}}h_y.$$

If $x,sx\in\text{Adm}(C)$, then without loss of generality, we may assume $x<sx\leq \lambda\in\pi(C)$. In this case,
\begin{align}\nonumber
T_{\tilde{s}}h_x+T_{\tilde{s}}h_{sx}
&=T_{\tilde{s}}r_{x,\lambda}(h_{\lambda})+T_{\tilde{s}}r_{sx,\lambda}(h_{\lambda})\\\nonumber
&=T_{\tilde{s}}r_{x,\lambda}(h_{\lambda})+c_{\tilde{s}}r_{sx,\lambda}(h_{\lambda})\\\nonumber
&=T_{\tilde{s}}r_{x,\lambda}(h_{\lambda})+T_{\tilde{s}}(-r_{x,sx}(r_{sx,\lambda}(h_{\lambda})))\\\nonumber
&=T_{\tilde{s}}r_{x,\lambda}(h_{\lambda})+T_{\tilde{s}}(-r_{x,\lambda}(h_{\lambda}))\\\nonumber
&=0.\nonumber
\end{align}
The second equality holds because $T_{\tilde{s}}T_{\tilde{s}\tilde{x}}=c_{\tilde{s}}T_{\tilde{s}\tilde{x}}$ for any $\tilde{x}\in W(1)$ with $\pi(\tilde{x})=x$. The third equality holds because $c_{\tilde{s}}T_{\tilde{s}}=T_{\tilde{s}}c_{\tilde{s}}$ and $c_{\tilde{s}}T_{\tilde{s}\tilde{x}}=T_{\tilde{s}}(c_{\tilde{s}}T_{\tilde{x}})=T_{\tilde{s}}(-r_{x,sx}(T_{\tilde{s}\tilde{x}}))$ for any $\tilde{x}\in W(1)$ with $\pi(\tilde{x})=x$. The fourth equality holds by Proposition \ref{2}. Therefore, 
$$T_{\tilde{s}}h_C=\sum_{x\in\text{Adm}(C),sx\notin\text{Adm}(C)}T_{\tilde{s}}h_x.$$

Similarly, 
$$h_CT_{\tilde{s}}=\sum_{x\in\text{Adm}(C),xs\notin\text{Adm}(C)}h_xT_{\tilde{s}}.$$

But it is easy to check by Lemma \ref{14} that there is a one-to-one correspondence between the two sets $\{x\in\text{Adm}(C)|sx\notin\text{Adm}(C)\}$ and $\{x\in\text{Adm}(C)|xs\notin\text{Adm}(C)\}$, i.e., $y\in\{x\in\text{Adm}(C)|sx\notin\text{Adm}(C)\}$ if and only if $sys\in\{x\in\text{Adm}(C)|xs\notin\text{Adm}(C)\}$. Therefore, it is enough to show that if $x\in\text{Adm}(C)$ and $sx\notin\text{Adm}(C)$, then
$$T_{\tilde{s}}h_x=h_{sxs}T_{\tilde{s}}.$$
Now $x<sx$, and we suppose $x\leq\lambda\in\pi(C)$. If $s\lambda>\lambda$, then by Lemma \ref{14} $sxs\leq s\lambda s$, thus
\begin{align}\nonumber
T_{\tilde{s}}h_x
&=T_{\tilde{s}}r_{x,\lambda}(h_{\lambda})\\\nonumber
&=r_{sx,s\lambda}(T_{\tilde{s}}h_{\lambda})\\\nonumber
&=r_{sx,s\lambda}(h_{sxs}T_{\tilde{s}})\\\nonumber
&=r_{sxs,s\lambda s}(h_{s\lambda s})T_{\tilde{s}}\\\nonumber
&=h_{sxs}T_{\tilde{s}}.\nonumber
\end{align}
The second and fourth equalities hold by Proposition \ref{2}. The third equality holds because $\tilde{s}\bullet h_{\lambda}=h_{s\lambda s}$.

If $s\lambda<\lambda$, then by Lemma \ref{14} $sx\leq\lambda$, but $\lambda<\lambda s$ so by Lemma \ref{14} again $sxs\leq\lambda$ and $sx\leq\lambda s$, therefore $x\leq s\lambda s$. Now let $y=sxs$, then $y\leq\lambda$ and $sys\leq s\lambda s$, therefore applying a similar proof as above, we have $h_yT_{\tilde{s}}=T_{\tilde{s}}h_{sys}$, i.e., $T_{\tilde{s}}h_x=h_{sxs}T_{\tilde{s}}$.

This finishes the proof.
\end{proof}

\begin{thm}
The center $\mathcal{Z}_R(0,c_{\tilde{s}})$ of $\mathcal{H}_R(0,c_{\tilde{s}})$ has a basis $\{h_C\}_{C\in\mathcal{F}(W(1))}$, where $\mathcal{F}(W(1))$ is the family of finite conjugacy classes in $W(1)$.
\end{thm}

\begin{proof}
First, $\{h_C\}_{C\in\mathcal{F}(W(1))}$ is linearly independent since $\text{supp}(h_C)_{\max}$ differs.

Next, we show that $\bigoplus_{C\in\mathcal{F}(W(1))}Rh_C = \mathcal{Z}_R(0,c_{\tilde{s}})$. We prove this by contradiction, so assume that $\bigoplus_{C\in\mathcal{F}(W(1))}Rh_C\subsetneqq \mathcal{Z}_R(0,c_{\tilde{s}})$.

Let $h$ be an element in $\mathcal{Z}_R(0,c_{\tilde{s}})-\bigoplus_{C\in\mathcal{F}(W(1))}Rh_C$ and $\max_{w\in\text{supp}(h)}\ell(w)\leq \max_{w\in \text{supp}(h')}\ell(w)$ for any $h'\in \mathcal{Z}_R(0,c_{\tilde{s}})-\bigoplus_{C\in\mathcal{F}(W(1))}Rh_C$.

Let $h=\sum_{\tilde{w}\in W(1)}a_{\tilde{w}}T_{\tilde{w}}$ and $M=\{\tilde{w}\in W(1)| \tilde{w} \ \ \text{maximal length with}\ \ a_{\tilde{w}}\neq 0\}$. By Theorem \ref{11}, $M$ is a union of some finite conjugacy class $C_i$'s. If $\tilde{w},\tilde{w}'\in C_i$ for some $i$, then $a_{\tilde{w}}= a_{\tilde{w}'}$, so we set $a_{C_i}=a_{\tilde{w}}$ for any $\tilde{w}\in C_i$. Let $h'=h-\sum_i a_{C_i} h_{C_i}$, then $h'\in \mathcal{Z}_R(0,c_{\tilde{s}})-\bigoplus_{C\in\mathcal{F}(W(1))}Rh_C$. But $\max_{w\in\text{supp}(h')}\ell(w)<\max_{w\in\text{supp}(h)}\ell(w)$. That is a contradiction.
\end{proof}

\section{Examples}
Given a finite conjugacy class $C$ in $W(1)$, we can write out the corresponding central element $h_C$ as follow. 

First we know $\pi(C)$, so we can write out $h_{\lambda,C}$ for each $\lambda\in\pi(C)$. For other $x\in\text{Adm}(C)$, it is easy to find a $\lambda\in\pi(C)$ such that $x<\lambda$. Then we can apply the operator $r_{x,\lambda}$ on $h_{\lambda,C}$ by changing some factors $T_{\tilde{s}}$ to $-c_{\tilde{s}}$. Adding up all these terms, we get $h_C$.

In this section, we give two examples to show how the above process works.

\begin{example}
In $GL_2$ case, the Iwahori Weyl group $W=W^{\text{aff}}\rtimes\Omega$. The affine Weyl group $W^{\text{aff}}$ is generated by $S^{\text{aff}}=\{s_0,s_1\}$. The group $\Omega$ is generated by $\tau$ and $\tau s_0=s_1\tau,\tau s_1=s_0\tau$.

Suppose $C_1$ is a finite conjugacy class in $W(1)$ with $$\pi(C_1)=\{s_0s_1s_0s_1,s_1s_0s_1s_0\}.$$ Then $\text{Adm}(C_1)=\{s_0s_1s_0s_1,s_1s_0s_1s_0,s_0s_1s_0,s_1s_0s_1,s_0s_1,s_1s_0,s_0,s_1,1\}$. 

Suppose 
$$h_{s_0s_1s_0s_1,C_1}=\sum_{t\in Z_1}T_{\tilde{s}_0\tilde{s}_1\tilde{s}_0\tilde{s}_1t},$$
for some subset $Z_1\subseteq Z$. Then
$$h_{s_1s_0s_1s_0,C_1}=\sum_{t\in Z_1}T_{\tilde{s}_1\tilde{s}_0\tilde{s}_1t\tilde{s}_0},$$
where $\tilde{s}_1\tilde{s}_0\tilde{s}_1t\tilde{s}_0$ is indeed a lifting of $s_1s_0s_1s_0$.

Since $s_0s_1s_0,s_1s_0s_1<s_0s_1s_0s_1$, we have
$$h_{s_0s_1s_0,C_1}=r_{s_0s_1s_0,s_0s_1s_0s_1}(h_{s_0s_1s_0s_1,C_1})=\sum_{t\in Z_1}-T_{\tilde{s}_0\tilde{s}_1\tilde{s}_0}c_{\tilde{s}_1t},$$
$$h_{s_1s_0s_1,C_1}=r_{s_1s_0s_1,s_0s_1s_0s_1}(h_{s_0s_1s_0s_1,C_1})=\sum_{t\in Z_1}-c_{\tilde{s}_0}T_{\tilde{s}_1\tilde{s}_0\tilde{s}_1t}.$$

Since $s_0s_1,s_1s_0<s_0s_1s_0s_1$, we have
$$h_{s_0s_1,C_1}=r_{s_0s_1,s_0s_1s_0s_1}(h_{s_0s_1s_0s_1,C_1})=\sum_{t\in Z_1}c_{\tilde{s}_0}c_{\tilde{s}_1}T_{\tilde{s}_0\tilde{s}_1t},$$
$$h_{s_1s_0,C_1}=r_{s_1s_0,s_0s_1s_0s_1}(h_{s_0s_1s_0s_1,C_1})=\sum_{t\in Z_1}c_{\tilde{s}_0}T_{\tilde{s}_1\tilde{s}_0}c_{\tilde{s}_1t}.$$

Since $s_0,s_1<s_0s_1s_0s_1$, we have
$$h_{s_0,C_1}=r_{s_0,s_0s_1s_0s_1}(h_{s_0s_1s_0s_1,C_1})=\sum_{t\in Z_1}-T_{\tilde{s}_0}c_{\tilde{s}_1}c_{\tilde{s}_0}c_{\tilde{s}_1t},$$
$$h_{s_1,C_1}=r_{s_1,s_0s_1s_0s_1}(h_{s_0s_1s_0s_1,C_1})=\sum_{t\in Z_1}-c_{\tilde{s}_0}c_{\tilde{s}_1}c_{\tilde{s}_0}T_{\tilde{s}_1t}.$$

Since $1<s_0s_1s_0s_1$, we have
$$h_{1,C_1}=r_{1,s_0s_1s_0s_1}(h_{s_0s_1s_0s_1,C_1})=\sum_{t\in Z_1}c_{\tilde{s}_0}c_{\tilde{s}_1}c_{\tilde{s}_0}c_{\tilde{s}_1t}.$$

We can easily tell that the parity of of sign is determined by length difference. 

Therefore the corresponding central element is 
\begin{align}\nonumber
h_{C_1}&=\sum_{t\in Z_1}T_{\tilde{s}_0\tilde{s}_1\tilde{s}_0\tilde{s}_1t}+T_{\tilde{s}_1\tilde{s}_0\tilde{s}_1t\tilde{s}_0}-T_{\tilde{s}_0\tilde{s}_1\tilde{s}_0}c_{\tilde{s}_1t}\\\nonumber
&-c_{\tilde{s}_0}T_{\tilde{s}_1\tilde{s}_0\tilde{s}_1t}
+c_{\tilde{s}_0}c_{\tilde{s}_1}T_{\tilde{s}_0\tilde{s}_1t}+c_{\tilde{s}_0}T_{\tilde{s}_1\tilde{s}_0}c_{\tilde{s}_1t}\\\nonumber
&-T_{\tilde{s}_0}c_{\tilde{s}_1}c_{\tilde{s}_0}c_{\tilde{s}_1t}-c_{\tilde{s}_0}c_{\tilde{s}_1}c_{\tilde{s}_0}T_{\tilde{s}_1t}+c_{\tilde{s}_0}c_{\tilde{s}_1}c_{\tilde{s}_0}c_{\tilde{s}_1t}.
\end{align}

Suppose $C_2$ is another finite conjugacy class in $W(1)$ with $$\pi(C_2)=\{s_0s_1s_0\tau,s_1s_0s_1\tau\}.$$ Then $\text{Adm}(C_2)=\{s_0s_1s_0\tau,s_1s_0s_1\tau,s_0s_1\tau,s_1s_0\tau,s_0\tau,s_1\tau,\tau\}$.

Suppose
$$h_{s_0s_1s_0\tau,C_2}=\sum_{t\in Z_2}T_{\tilde{s}_0\tilde{s}_1\tilde{s}_0\tilde{\tau}t},$$
for some subset $Z_2\subseteq Z$. Then
$$h_{s_1s_0s_1\tau,C_2}=\sum_{t\in Z_2}T_{\tilde{s}_1\tilde{s}_0\tilde{s}_1\tilde{s}_0\tilde{\tau}t\tilde{s}_1^{-1}}=\sum_{t\in Z_2}T_{\tilde{s}_1\tilde{s}_0\tilde{s}_1\tilde{\tau}(\tilde{\tau}^{-1}\tilde{s}_0\tilde{\tau}t\tilde{s}_1^{-1})},$$
where $(\tilde{\tau}^{-1}\tilde{s}_0\tilde{\tau}t\tilde{s}_1^{-1})$ is an element in $Z$. So $\tilde{s}_1\tilde{s}_0\tilde{s}_1\tilde{\tau}(\tilde{\tau}^{-1}\tilde{s}_0\tilde{\tau}t\tilde{s}_1^{-1})$ is indeed a lifting of $s_1s_0s_1\tau$.

Since $s_0s_1\tau,s_1s_0\tau<s_0s_1s_0\tau$, we have
$$h_{s_0s_1\tau,C_2}=r_{s_0s_1\tau,s_0s_1s_0\tau}(h_{s_0s_1s_0\tau,C_2}),\quad h_{s_1s_0\tau,C_2}=r_{s_1s_0\tau,s_0s_1s_0\tau}(h_{s_0s_1s_0\tau,C_2}).$$

Since $s_0\tau,s_1\tau<s_0s_1s_0\tau$, we have
$$h_{s_0\tau,C_2}=r_{s_0\tau,s_0s_1s_0\tau}(h_{s_0s_1s_0\tau,C_2}),\quad h_{s_1\tau,C_2}=r_{s_1\tau,s_0s_1s_0\tau}(h_{s_0s_1s_0\tau,C_2}).$$

Since $\tau<s_0s_1s_0\tau$, we have
$$h_{\tau,C_2}=r_{\tau,s_0s_1s_0\tau}(h_{s_0s_1s_0\tau,C_2}).$$

Therefore the corresponding central element is 
\begin{align}\nonumber
h_{C_2}&=\sum_{t\in Z_2}T_{\tilde{s}_0\tilde{s}_1\tilde{s}_0\tilde{\tau}t}+T_{\tilde{s}_1\tilde{s}_0\tilde{s}_1\tilde{\tau}(\tilde{\tau}^{-1}\tilde{s}_0\tilde{\tau}t\tilde{s}_1^{-1})}-T_{\tilde{s}_0\tilde{s}_1}c_{\tilde{s}_0}T_{\tilde{\tau}t}-c_{\tilde{s}_0}T_{\tilde{s}_1\tilde{s}_0\tilde{\tau}t}\\\nonumber
&+c_{\tilde{s}_0}c_{\tilde{s}_1}T_{\tilde{s}_0\tilde{\tau}t}+c_{\tilde{s}_0}T_{\tilde{s}_1}c_{\tilde{s}_0}T_{\tilde{\tau}t}-c_{\tilde{s}_0}c_{\tilde{s}_1}c_{\tilde{s}_0}T_{\tilde{\tau}t}.
\end{align}
\end{example}

\begin{example}
In $SL_3$ case, the Iwahori Weyl group $W=W^{\text{aff}}$. The affine Weyl group $W^{\text{aff}}$ is generated by $S^{\text{aff}}=\{s_0,s_1,s_2\}$ with braid relations $s_is_js_i=s_js_is_j$ for $i\neq j$.

Suppose $C$ is a finite conjugacy class in $W(1)$ with $$\pi(C)=\{s_0s_1s_2s_1,s_1s_0s_1s_2,s_2s_0s_2s_1,s_1s_2s_1s_0,s_2s_1s_0s_1,s_1s_2s_0s_2\}.$$ Then $$\text{Adm}(C)=\{s_0s_1s_2s_1,s_1s_0s_1s_2,s_2s_0s_2s_1,s_1s_2s_1s_0,s_2s_1s_0s_1,s_1s_2s_0s_2,$$$$s_1s_2s_1,s_1s_0s_1,s_2s_0s_2,s_0s_1s_2,s_0s_2s_1,s_1s_0s_2,s_1s_2s_0,s_2s_1s_0,s_2s_0s_1,$$$$s_0s_1,s_0s_2,s_1s_2,s_2s_1,s_1s_0,s_2s_0,s_0,s_1,s_2,1\}.$$

Suppose $$h_{s_0s_1s_2s_1,C}=\sum_{t\in Z'}T_{\tilde{s}_0\tilde{s_1}\tilde{s}_2\tilde{s}_1t},$$
for some subset $Z'\subseteq Z$. Then
\begin{align}\nonumber
h_{s_1s_0s_1s_2,C}&=\sum_{t\in Z'}T_{\tilde{s}_1t\tilde{s}_0\tilde{s_1}\tilde{s}_2},\\\nonumber
h_{s_2s_0s_2s_1,C}&=\sum_{t\in Z'}T_{\tilde{s}_2\tilde{s}_0\tilde{s_1}\tilde{s}_2\tilde{s}_1t\tilde{s}_2^{-1}}=\sum_{t\in Z'}T_{\tilde{s}_2\tilde{s}_0\tilde{s}_2\tilde{s}_1(\tilde{s}_1^{-1}\tilde{s}_2^{-1}\tilde{s}_1\tilde{s}_2\tilde{s}_1t\tilde{s}_2^{-1})},\\\nonumber
h_{s_1s_2s_1s_0,C}&=\sum_{t\in Z'}T_{\tilde{s_1}\tilde{s}_2\tilde{s}_1t\tilde{s}_0},\\\nonumber
h_{s_2s_1s_0s_1,C}&=\sum_{t\in Z'}T_{\tilde{s}_2\tilde{s}_1t\tilde{s}_0\tilde{s_1}},\\\nonumber
h_{s_1s_2s_0s_2,C}&=\sum_{t\in Z'}T_{\tilde{s}_2^{-1}\tilde{s_1}\tilde{s}_2\tilde{s}_1t\tilde{s}_0\tilde{s}_2}=\sum_{t\in Z'}T_{(\tilde{s}_2^{-1}\tilde{s}_1\tilde{s}_2\tilde{s}_1t\tilde{s}_2^{-1}\tilde{s}_1^{-1})\tilde{s}_1\tilde{s}_2\tilde{s}_0\tilde{s}_2},\nonumber
\end{align}
where $\tilde{s}_1^{-1}\tilde{s}_2^{-1}\tilde{s}_1\tilde{s}_2\tilde{s}_1t\tilde{s}_2^{-1},\tilde{s}_2^{-1}\tilde{s}_1\tilde{s}_2\tilde{s}_1t\tilde{s}_2^{-1}\tilde{s}_1^{-1}$ are elements in $Z$. So the elements $\tilde{s}_2\tilde{s}_0\tilde{s}_2\tilde{s}_1(\tilde{s}_1^{-1}\tilde{s}_2^{-1}\tilde{s}_1\tilde{s}_2\tilde{s}_1t\tilde{s}_2^{-1})$ and $(\tilde{s}_2^{-1}\tilde{s}_1\tilde{s}_2\tilde{s}_1t\tilde{s}_2^{-1}\tilde{s}_1^{-1})\tilde{s}_1\tilde{s}_2\tilde{s}_0\tilde{s}_2$ are indeed liftings of $s_2s_0s_2s_1$ and $s_1s_2s_0s_2$ respectively.

Since $s_1s_2s_1,s_0s_1s_2,s_0s_2s_1<s_0s_1s_2s_1;s_1s_0s_1,s_1s_0s_2<s_1s_0s_1s_2;s_2s_0s_2<s_2s_0s_2s_1;s_1s_2s_0,s_2s_1s_0<s_1s_2s_1s_0;s_2s_0s_1<s_2s_1s_0s_1$, we have 
$$h_{s_1s_2s_1,C}=r_{s_1s_2s_1,s_0s_1s_2s_1}(h_{s_0s_1s_2s_1,C}),h_{s_1s_0s_1,C}=r_{s_1s_0s_1,s_1s_0s_1s_2}(h_{s_1s_0s_1s_2,C}),$$
$$h_{s_2s_0s_2,C}=r_{s_2s_0s_2,s_2s_0s_2s_1}(h_{s_2s_0s_2s_1,C}),h_{s_0s_1s_2,C}=r_{s_0s_1s_2,s_0s_1s_2s_1}(h_{s_0s_1s_2s_1,C}),$$
$$h_{s_0s_2s_1,C}=r_{s_0s_2s_1,s_0s_1s_2s_1}(h_{s_0s_1s_2s_1,C}),h_{s_1s_0s_2,C}=r_{s_1s_0s_2,s_1s_0s_1s_2}(h_{s_1s_0s_1s_2,C}),$$
$$h_{s_1s_2s_0,C}=r_{s_1s_2s_0,s_1s_2s_1s_0}(h_{s_1s_2s_1s_0,C}),h_{s_2s_1s_0,C}=r_{s_2s_1s_0,s_1s_2s_1s_0}(h_{s_1s_2s_1s_0,C}),$$
$$h_{s_2s_0s_1,C}=r_{s_2s_0s_1,s_2s_1s_0s_1}(h_{s_2s_1s_0s_1,C}).$$

Since $s_0s_1,s_0s_2,s_1s_2,s_2s_1<s_0s_1s_2s_1;s_1s_0,s_2s_0<s_1s_2s_1s_0$, we have
$$h_{s_0s_1,C}=r_{s_0s_1,s_0s_1s_2s_1}(h_{s_0s_1s_2s_1,C}),h_{s_0s_2,C}=r_{s_0s_2,s_0s_1s_2s_1}(h_{s_0s_1s_2s_1,C}),$$
$$h_{s_1s_2,C}=r_{s_1s_2,s_0s_1s_2s_1}(h_{s_0s_1s_2s_1,C}),h_{s_2s_1,C}=r_{s_2s_1,s_0s_1s_2s_1}(h_{s_0s_1s_2s_1,C}),$$
$$h_{s_1s_0,C}=r_{s_1s_0,s_1s_2s_1s_0}(h_{s_1s_2s_1s_0,C}),h_{s_2s_0,C}=r_{s_2s_0,s_1s_2s_1s_0}(h_{s_1s_2s_1s_0,C}).$$

Since $s_0,s_1,s_2<s_0s_1s_2s_1$, we have
$$h_{s_0,C}=r_{s_0,s_0s_1s_2s_1}(h_{s_0s_1s_2s_1,C}),h_{s_1,C}=r_{s_1,s_0s_1s_2s_1}(h_{s_0s_1s_2s_1,C}),$$
$$h_{s_2,C}=r_{s_2,s_0s_1s_2s_1}(h_{s_0s_1s_2s_1,C}).$$

Since $1<s_0s_1s_2s_1$, we have
$$h_{1,C}=r_{1,s_0s_1s_2s_1}(h_{s_0s_1s_2s_1,C}).$$

Therefore the corresponding central element is
\begin{align}\nonumber
h_C&=\sum_{t\in Z'}T_{\tilde{s}_0\tilde{s}_1\tilde{s}_2\tilde{s}_1t}+T_{\tilde{s}_1t\tilde{s}_0\tilde{s_1}\tilde{s}_2}+T_{\tilde{s}_2\tilde{s}_0\tilde{s}_2\tilde{s}_1(\tilde{s}_1^{-1}\tilde{s}_2^{-1}\tilde{s}_1\tilde{s}_2\tilde{s}_1t\tilde{s}_2^{-1})}\\\nonumber
&+T_{\tilde{s}_1\tilde{s}_2\tilde{s}_1t\tilde{s}_0}+T_{\tilde{s}_2\tilde{s}_1t\tilde{s}_0\tilde{s}_1}+T_{(\tilde{s}_2^{-1}\tilde{s}_1\tilde{s}_2\tilde{s}_1t\tilde{s}_2^{-1}\tilde{s}_1^{-1})\tilde{s}_1\tilde{s}_2\tilde{s}_0\tilde{s}_2}\\\nonumber
&-c_{\tilde{s}_0}T_{\tilde{s}_1\tilde{s}_2\tilde{s}_1t}-T_{\tilde{s}_1t\tilde{s}_0\tilde{s_1}}c_{\tilde{s}_2}-T_{\tilde{s}_2\tilde{s}_0\tilde{s}_2}c_{\tilde{s}_1(\tilde{s}_1^{-1}\tilde{s}_2^{-1}\tilde{s}_1\tilde{s}_2\tilde{s}_1t\tilde{s}_2^{-1})}\\\nonumber
&-T_{\tilde{s}_0\tilde{s}_1\tilde{s}_2}c_{\tilde{s}_1t}-T_{\tilde{s}_0}c_{\tilde{s}_1}T_{\tilde{s}_2\tilde{s}_1t}-T_{\tilde{s}_1t\tilde{s}_0}c_{\tilde{s_1}}T_{\tilde{s}_2}\\\nonumber
&-T_{\tilde{s}_1\tilde{s}_2}c_{\tilde{s}_1t}T_{\tilde{s}_0}-c_{\tilde{s}_1}T_{\tilde{s}_2\tilde{s}_1t\tilde{s}_0}-T_{\tilde{s}_2}c_{\tilde{s}_1t}T_{\tilde{s}_0\tilde{s}_1}\\\nonumber
&+T_{\tilde{s}_0\tilde{s}_1}c_{\tilde{s}_2}c_{\tilde{s}_1t}+T_{\tilde{s}_0}c_{\tilde{s}_1}T_{\tilde{s}_2}c_{\tilde{s}_1t}+c_{\tilde{s}_0}T_{\tilde{s}_1\tilde{s}_2}c_{\tilde{s}_1t}+c_{\tilde{s}_0}c_{\tilde{s}_1}T_{\tilde{s}_2\tilde{s}_1t}\\\nonumber
&+T_{\tilde{s}_1}c_{\tilde{s}_2}c_{\tilde{s}_1t}T_{\tilde{s}_0}+c_{\tilde{s}_1}T_{\tilde{s}_2}c_{\tilde{s}_1t}T_{\tilde{s}_0}-T_{\tilde{s}_0}c_{\tilde{s}_1}c_{\tilde{s}_2}c_{\tilde{s}_1t}\\\nonumber
&-c_{\tilde{s}_0}c_{\tilde{s}_1}c_{\tilde{s}_2}T_{\tilde{s}_1t}-c_{\tilde{s}_0}c_{\tilde{s}_1}T_{\tilde{s}_2}c_{\tilde{s}_1t}+c_{\tilde{s}_0}c_{\tilde{s}_1}c_{\tilde{s}_2}c_{\tilde{s}_1t}.
\end{align}
\end{example}

\end{document}